\documentclass[a4paper]{article}
\usepackage{amssymb}
\usepackage{graphicx,color}
\usepackage{amsthm}
\usepackage{amsmath}
\usepackage{enumerate}
\usepackage{tikz}

\usepackage{float}
\usepackage{subfig}
\usepackage{hyperref}
\usepackage{multirow}
\usepackage{rotating}
\usepackage{bbm}

\theoremstyle{plain}
\newtheorem{theorem}{Theorem}
\newtheorem{proposition}[theorem]{Proposition}

\newtheorem{lemma}[theorem]{Lemma}

\theoremstyle{definition}

\newtheorem{definition}[theorem]{Definition}

\def\ho{\H o}

\def\ra{\rightarrow}

\def\ba{\begin{array}}
\def\ea{\end{array}}
\def\bi{\begin{itemize}}
\def\ei{\end{itemize}}
\def\mR{\mathbb{R}}

\def\mE{\mathbb{E}}

\def\m1{1}

\def\cZ{\mathcal{Z}}
\def\cN{\mathcal{N}}
\def\cL{\mathcal{L}}

\def\eps{\varepsilon}

\let\oldhat\hat

\renewcommand{\hat}[1]{\oldhat{\mathbf{#1}}}

\providecommand{\com[2]}{{\color{blue} \begin{tt}[#1: #2]\end{tt}}}

\begin{document}
\title{Mixing time of an unaligned Gibbs sampler on the square}
\author{Bal\'azs Gerencs\'er\thanks{B. Gerencs\'er is with the Alfr\'ed R\'enyi Institute
    of Mathematics, Hungarian Academy of Sciences and
the ELTE E\"otv\"os Lor\'and University, Department of Probability Theory and Statistics,
  {\tt\small gerencser.balazs@renyi.mta.hu}. His work is supported by NKFIH (National Research, Development and Innovation Office) grant PD 121107.}%
}
\date{\today}

\maketitle

\begin{abstract}
  The paper concerns a particular example of the Gibbs sampler
  and its mixing efficiency. Coordinates of a point are rerandomized
  in the unit square $[0,1]^2$ to approach a stationary
  distribution with density proportional to $\exp(-A^2(u-v)^2)$ for $(u,v)\in [0,1]^2$ with some
  large parameter $A$.

  Diaconis conjectured the mixing time of this process to be $O(A^2)$ which
  we confirm in this paper. This improves on the
  currently known $O(\exp(A^2))$ estimate.
\end{abstract}



\section{Introduction}

A standard use of Markov chains is to sample from a probability distribution
that would be otherwise hard to access.
This can happen when the
distribution is supported on a set implicitly defined by some constraints, e.g., a
convex body in a high dimensional space
\cite{kannan:convexvolume1997}, \cite{lovasz:convexvolume2006}, proper
colorings of a graph \cite{dyer:randomcolorrandomgraph2006},
\cite{mossel:gibbs_erdosrenyi2010}, etc.
Several frameworks have been
designed to achieve this goal including the Metropolis algorithm and
the Gibbs sampler and their variants.
There is a vast range of
applications and studies, we refer the reader to \cite{diaconis:metropolis1998}, \cite{diaconis:montecarlo_broadsurvey2009}
for orientation.

A central and recurring question is the efficiency of these algorithms
in the different settings. We highlight two phenomena that can
decrease the performance of such algorithms. First, the incremental
change the Markov chain allows is usually quite rigid and given by
the structure of the state space. However, the desired stationary
distribution does not need to be aligned with the directions where the
Markov chain mixes fast. Second, some boundary effects might occur if
the Markov chain can get trapped in some remote part of the state space.

In this paper we analyze an example of the Gibbs sampling procedure
proposed by Diaconis which is surprisingly simple considering it
captures both of the two phenomena above. We call the \emph{coordinate
  Gibbs sampler for the diagonal distribution} the following process. Fix a
large positive
constant $A$ and on $[0,1]^2$ define the distribution $\pi$ with
density proportional to $\exp(-A^2(u-v)^2)$ for $(u,v)\in [0,1]^2$. At each step randomly
choose coordinate $u$ or $v$ and rerandomize it according to the
conditional distribution of $\pi$.
Notice that the distribution of this Markov chain is mostly concentrated near
the diagonal of the unit square, while only horizontal and vertical
transitions are allowed.
Furthermore, near $(0,0)$ and $(1,1)$ we
see that both the high density of $\pi$ and also the boundaries of the
square hinder the movement of the chain.

The efficiency of the algorithm is quantified by the mixing
time of the Markov chain. For any Markov chain
$X(0),X(1),\ldots$ on some state space $\Omega$ (which is $[0,1]^2$ in
our case) let $\cL(X(t))$ denote the distribution of the
state at time $t$ and $\eta$ be the stationary distribution assuming
it is unique (denoted
by $\pi$ for our case). Using
the total variation distance between
measures, $\|\rho-\sigma\|_{\rm TV} := \sup_{S\subseteq \Omega}
|\rho(S)-\sigma(S)|$ we define the mixing time as
$$t_{\rm mix}(X,\eps):=\sup_{X(0)\in \Omega}\min\left\{t:\|\cL(X(t))-\eta\|_{{\rm
      TV}} \le \eps\right\}.$$
Diaconis conjectured that the
mixing time of the example proposed is $O(A^2)$, the goal of this
paper is to confirm this bound.

\begin{theorem}
  \label{thm:Xmix}
  Let $X(t)$ follow the coordinate Gibbs sampler for the diagonal
  distribution. For any $0< \alpha <1$ there exists $\beta > 0$ such that for large
  enough $A$,
  $$t_{\rm mix}(X, \alpha) \le \beta A^2.$$
\end{theorem}

Up until now only $O(\exp(A^2))$ was
known which easily follows from a minorization condition of the
transition kernel.

Observe that the diagonal nature of the distribution plays an
important role in the mixing behavior, making the distribution and the randomization steps
unaligned. If we took the distribution with density
proportional to $\exp(A^2(u - 1/2)^2)$ for $(u,v)\in [0,1]^2$, then the
mixing time would decrease to be $O(1)$. Indeed, this is a product
distribution, product of one for $u$ and one (uniform) for $v$, consequently after
a rerandomization is performed along both coordinates, the
distribution of the process will exactly match the prescribed one. This will happen
with probability arbitrarily close to 1 within a corresponding finite
number of steps, not depending on the value of $A$.

The rest of the paper is organized as follows.
In Section \ref{sec:preliminaries} a formal definition of the
process of interest is provided and further variants are introduced that help the analysis. Section
\ref{sec:dyn_Yu} provides the building blocks for the proof, to
understand the short-term behavior of the process based on the
initialization. Afterwards, the proof of Theorem \ref{thm:Xmix} is aggregated in
Section \ref{sec:mixproof}. Finally, a complementing lower bound
demonstrating that Theorem \ref{thm:Xmix} is essentially sharp is
given in Section \ref{sec:others} together with some numerical simulations.

\section{Preliminaries, alternative processes}
\label{sec:preliminaries}

We now formally define the
\emph{coordinate Gibbs sampler for the diagonal distribution}
which we denote by $X(t)$, then we introduce
variants that will be more convenient to handle.

Let $\varphi(x):=\exp(-A^2x^2)$ for some large $A>0$ and
let $\pi$ be the probability distribution on $[0,1]^2$ with density $\cZ^{-1}\varphi(u-v)$
at $(u,v)\in [0,1]^2$ (where $\cZ=\int_{[0,1]^2}\varphi(u-v)$).
We write $\pi(\cdot,v)$ for the conditional distribution of the $u$ coordinate
when $v$ is fixed (similarly for $\pi(u,\cdot)$). Denote by $\pi_u$
the projection of $\pi$, that is, the overall distribution of the $u$ coordinate.

When defining the coordinate Gibbs sampler for the diagonal distribution, we separate the decision of the direction of
randomization and the randomization itself.
For $t=1,2,\ldots$ let $r(t)$ be an i.i.d.\ sequence of variables of
characters $U,V$ taking both with probability $1/2$.
Given some initial point $X(0)\in [0,1]^2$ the random variable $X(t) =
(X_u(t),X_v(t))$ is
generated as a Markov chain from $X(t-1)$ by randomizing
along the axis given by $r(t)$. Formally,
$$
X(t) :=
\begin{cases}
  \big(u^+,X_v(t-1)\big), & \text{ if } r(t)=U, \text{where } u^+\sim \pi(\cdot,X_v(t-1)), \\
  \big(X_u(t-1), v^+\big), & \text{ if } r(t)=V, \text{where } v^+\sim \pi(X_u(t-1),\cdot),
\end{cases},
$$
where $u^+, v^+$ are conditionally independent of the past at all steps.

Note that when multiple $U$'s follow each other in the series $r(t)$
(similarly for $V$), the values $u^+$ are repeatedly overwritten and
forgotten, with no further mixing happening for the overall
distribution. Therefore we define an alternative process where this effect
does not occur, but rather the directions of randomization
are deterministic.

Let $X^*(0):=X(0)$, then the following process is generated:
\begin{alignat*}{3}
  X^*(2s) &:= \big(u^+,X^*_v(2s-1)\big), &\qquad \text{where } u^+ &\sim \pi(\cdot,X^*_v(2s-1)),\\
  X^*(2s+1) &:= \big(X^*_u(2s), v^+\big), & \text{where } v^+& \sim \pi(X^*_u(2s),\cdot).
\end{alignat*}

It would be convenient for the analysis if it wasn't necessary to
distinguish the steps based on the parity of the time index.
For that reason, consider the following modification. At every even
step take $X^*(2s)$ as before, at every odd step take $X^*(2s+1)$ flipped along the
diagonal of the square (exchange the two coordinates). Equivalently,
flip the process at every step while generating. As a result, the
randomization happens in the same direction at every step. 
Note that the target distribution $\pi$ is
symmetric along the diagonal therefore no adjustment is needed for the flipping.
Formally, the process described is the following:

Let $Y(0):=X(0)$, then the random variables $Y(t)$ are generated from
$Y(t-1)$ as follows
$$Y(t) := \big(u^+,Y_u(t-1)\big), \qquad \text{where } u^+ \sim \pi(\cdot,Y_u(t-1)).$$

Observe that the scalar process $Y_u(t)$ is a Markov chain by itself
simply because $Y(t)$ depends on $Y(t-1)$ only through $Y_u(t-1)$.

\section{Dynamics of $Y_u(t)$}
\label{sec:dyn_Yu}

In this section we prove two properties of the evolution of $Y_u(t)$,
which will be the key elements to compute the mixing time bounds. First, we show that the process cannot stay arbitrarily long
at the sides of the unit interval, in $[0,1/2-\delta)$ or
$(1/2+\delta,1]$, where some small enough parameter $\delta>0$ will be
chosen. Second, we prove that starting from a point in the
middle part $[1/2-\delta,1/2+\delta]$, the distribution of the
process quickly approaches the stationary distribution.

\subsection{Reaching the middle}
\label{sec:middle}

We work on the case when the $Y_u(0)$ is away from the
middle of $[0,1]$. 
We want to ensure that the process does not stay near the boundaries
for a long period. To quantify this, the time to reach
the middle is defined as follows:

\begin{definition}
  Let $\nu_m := \min \{s: Y_u(s) \in[1/2-\delta, 1/2+\delta]\}$.
\end{definition}

Without the loss of generality we may assume that $Y_u(0)$ is on the
left part of $[0,1]$, thanks to the symmetry of $\pi$
w.r.t.\ $(1/2, 1/2)$. Therefore we start from $Y_u(0)<1/2-\delta$.
For this period before reaching the middle we introduce a slightly simplified process $Y'$, where
both coordinates are allowed to take values in $[0,\infty)$ in
principle. This is not supposed to have a substantially different behavior, but will allow more convenient
analytic investigation as fewer boundaries are present.

For any $v\in\mR$ let $\sigma_v$ be the measure on $[0,\infty)$ with density proportional to
$\varphi(u-v)$ conditioned on $u\in [0,\infty)$. Let
$Y'(0):=X(0)$, then define the Markov chain $Y'(t)$ as follows:
$$Y'(t) := \big(u^+,Y'_u(t-1)\big), \qquad \text{where } u^+ \sim \sigma_{Y'_u(t-1)}.$$

We can generate $Y'(t)$ to be coupled to $Y(t)$ as long as possible. For a
fixed $v$, $\pi(u,v)$ is proportional to $\varphi(u-v)$ conditioned
on $u\in [0,1]$. Therefore, when we need to generate $u^+$ we draw a
random sample from $\sigma_{Y'_u(t-1)}$ and use it for both $Y(t)$ and
$Y'(t)$ if $u^+<1$. Otherwise, we use it for $Y'(t)$ but for $Y(t)$ we draw a new independent sample from
$\pi(\cdot,Y'_u(t-1))$. It is easy to verify this is overall a valid method for generating a
random variable of distribution $\pi(\cdot,Y'_u(t-1))$.

In the latter case, we also signal decoupling by setting a stopping time
$\nu_c^1=t$. We show this rarely happens, when governed by a variant
of $\nu_m$.
Let $\tilde{\nu}_m := \min \{s: Y_u(s) \ge 1/2-\delta\}$.

\begin{lemma}
  \label{lm:Yprime_coupling}
  For any $\alpha_1>0$ there is $\beta_1>0$ such that $P(\nu_c^1 < \min(\tilde{\nu}_m,
  \alpha_1 A^2)) = O(\exp(-\beta_1 A^2))$.
\end{lemma}
\begin{proof}
  We want to bound the probability of decoupling at every point in time.

  When $u^+$ is drawn, $Y'_u(t-1)<1/2-\delta$ is ensured as
  $\tilde{\nu}_m$ has not yet occurred. For any $v<1/2-\delta$ we have
  $$\sigma_v(\{u^+>1\})\le 2 P(u > 1, u\sim \cN(v,1/(2A^2)))\le
  2\frac{\exp(-A^2(1/2+\delta)^2)}{2\sqrt{\pi}A(1/2+\delta)}.$$
  Here we use that the conditional probability is at most twice the
  unconditional one (because of $v\ge 0$), use the monotonicity in $v$, then apply a standard tail
  probability estimate for the Gaussian distribution.

  These exceptional events may occur at most at $\alpha_1 A^2$ different
  times, therefore by using the union bound the overall probability is
  $O(\exp(-\beta_1 A^2))$ for any $\beta_1< (1/2+\delta)^2$.
\end{proof}

\begin{lemma}
  \label{lm:nu_m_variants}
  There exists $\beta_2>0$ constant such that $P(\nu_m \neq
  \tilde{\nu}_m) = O(\exp(-\beta_2 A^2)$.
\end{lemma}
\begin{proof}
  By a similar argument as above this bad event $\{\nu_m \neq
  \tilde{\nu}_m\}$ happens when $Y_u(t-1)<1/2-\delta$ but
  $Y_u(t)>1/2+\delta$ when $\tilde{\nu}_m$ occurs, then a Gaussian tail probability estimate gives
  an upper bound of
  $$2\frac{\exp(-A^2(2\delta)^2)}{2\sqrt{\pi}A(2\delta)}.$$
  The lemma holds with $\beta_2=(2\delta)^2$.
\end{proof}

Handling $Y'(t)$ is still challenging due to the conditional
distributions included in the definition. Therefore we introduce the
following process that will be both convenient to handle and
to relate to $Y'(t)$.

Let $\tilde{Z}(t)$ be a random walk with i.i.d.\
$\cN(0,1/(2A^2))$ increments, starting from $\tilde{Z}(0) := X_u(0)$.\\
Let $Z(t) := |\tilde Z(t)|$.

Let us denote by $\phi$ the distribution of the centered Gaussian with
variance $1/(2A^2)$.
During the analysis of $Z(t)$ we will also need to use the distribution of
the absolute value of a Gaussian distribution with variance $1/(2A^2)$. We denote it by $\phi_x$ when the original one is
centered at $x$ and it is easy to verify that we can express it for any
$A\subset [0,\infty)$ by $\phi_x(A) = \phi(A-x)+\phi(-A-x)$.

\begin{proposition}
  \label{prp:couple_yprime_z}
  $Z(t)$ and $Y'_u(t)$ can be coupled such that $Z(t)\le Y'_u(t)$ for all $t\ge 0$.
\end{proposition}

\begin{proof}
  At 0 we have $Z(0) = Y'_u(0)$. We construct the coupling iteratively,
  assuming $Z(t-1) \le Y'_u(t-1)$ we perform the next step of the
  coupling which will satisfy $Z(t) \le Y'_u(t)$.

  We will use the monotone coupling between the two. For two probability
  distributions $\rho,\rho'$ the monotone coupling is the one
  assigning $x$ to $x'$ when $\rho((-\infty,x])=\rho'((-\infty,x'])$.
  (We now skip currently irrelevant technical details about continuity,
  etc.). It is easy to verify that $x\le x'$ is
  maintained through this coupling exactly if $\rho((-\infty,y]) \ge
  \rho'((-\infty,y])$ for all $y$. In our case we will need the following: 
  \begin{lemma}
    \label{lm:comp_mon_coupling}
    For any $v\ge \bar{v}\ge 0$ and $u\ge 0$:
    $$\phi_{\bar{v}}([0,u])\ge \sigma_v([0,u]).$$
  \end{lemma}
  Here $\bar{v}$ corresponds to $Z(t-1)$ and $v$ to $Y'_u(t-1)$ and we
  compare the distributions for step $t$.
    
  \begin{proof}
    We are going to prove the following two inequalities:
    $$\phi_{\bar{v}}([0,u])\ge \phi_v([0,u]), \qquad \phi_v([0,u])\ge \sigma_v([0,u]).$$

    For the first of the two we compute $\partial_v \phi_v([0,u])$:
    $$\partial_v \phi_v([0,u]) = \partial_v\left(\phi([-v-u,-v+u])\right)$$
    $$ = \partial_v \left(\frac{1}{\int_{-\infty}^\infty\varphi} \int\limits_{-v-u}^{-v+u} \varphi\right)$$
    $$ = \frac{1}{\int_{-\infty}^\infty\varphi}(-\varphi(-v+u)+\varphi(-v-u))\le 0.$$
    This last inequality holds because $|-v+u|\le |-v-u|$ and
    $\varphi(x)$ is decreasing in $|x|$. Consequently, when $\bar{v}$ is
    increased to $v$, the measure of $[0,u]$ decreases confirming the
    first inequality. 
    This intuitively means that when a Gaussian distribution is shifted to
    the right then even the reflected Gaussian is shifted (if it was
    centered at a non-negative point).
    
    The second inequality to confirm is the following:
    $$\phi_v([0,u]) = \phi([-v-u,-v+u]) \ge
    \sigma_v([0,u]).$$
    We rearrange and cancel out as much as possible from the domain of integrations.
    \begin{align*}
      \left. \int_{-v-u}^{-v+u}\varphi \middle/ \int_{-\infty}^\infty \varphi \right. &\ge
  \left. \int_{-v}^{-v+u}\varphi \middle/ \int_{-v}^\infty \varphi \right.\\
  \int_{-v-u}^{-v+u}\varphi \cdot \int_{-v}^\infty \varphi &\ge
  \int_{-v}^{-v+u}\varphi \cdot \int_{-\infty}^\infty \varphi\\
  \left(\int_{-v-u}^{-v}\varphi + \int_{-v}^{-v+u}\varphi\right) \cdot \int_{-v}^\infty \varphi &\ge
  \int_{-v}^{-v+u}\varphi \cdot \left(\int_{-\infty}^{-v} \varphi + \int_{-v}^\infty \varphi\right)\\
  \int_{-v-u}^{-v}\varphi \cdot \int_{-v}^\infty \varphi &\ge
  \int_{-v}^{-v+u}\varphi \cdot \int_{-\infty}^{-v} \varphi\\
  \int_{-v-u}^{-v}\varphi \cdot \left(\int_{-v}^{-v+u} \varphi + \int_{-v+u}^\infty \varphi \right) &\ge
  \int_{-v}^{-v+u}\varphi \cdot \left(\int_{-\infty}^{-v-u} \varphi + \int_{-v-u}^{-v} \varphi\right)\\
  \int_{-v-u}^{-v}\varphi \cdot \int_{-v+u}^\infty \varphi &\ge
  \int_{-v}^{-v+u}\varphi \cdot \int_{-\infty}^{-v-u} \varphi
    \end{align*}

  We substitute the functions to integrate and transform them to compare them on
  the same domain.
\begin{align*}
  \int \limits_{-v-u}^{-v}e^{-A^2x^2}dx \cdot \int \limits_{-v+u}^\infty e^{-A^2y^2}dy &\ge
  \int \limits_{-v}^{-v+u}e^{-A^2x^2}dx \cdot \int \limits_{-\infty}^{-v-u} e^{-A^2y^2}dy\\
  \int \limits_{0}^{u}e^{-A^2(x+v)^2}dx \cdot \int \limits_{u}^\infty e^{-A^2(y-v)^2}dy &\ge
  \int \limits_{0}^{u}e^{-A^2(x-v)^2}dx \cdot \int \limits_{u}^\infty e^{-A^2(y+v)^2}dy\\
  \int \limits_{0}^{u} \int \limits_{u}^\infty e^{-A^2(x^2+y^2+2v^2 + 2v(x-y))} dy dx &\ge
  \int \limits_{0}^{u} \int \limits_{u}^\infty e^{-A^2(x^2+y^2+2v^2 - 2v(x-y))} dy dx
\end{align*}

  On all the domain of integration we have $x\le y$. Therefore the
  exponent is larger at every point for the left hand side, which
  confirms the second inequality, completing the proof of the lemma.
\end{proof}

Lemma \ref{lm:comp_mon_coupling} thus ensures that the monotone
coupling preserves the ordering, and we can indeed use the recursive
coupling scheme while keeping $Z(t)\le Y'_u(t)$ at every step.
\end{proof}

\begin{proposition}
  \label{prp:nu_m_bound}
  For any $\alpha_3>0$ there exists $\beta_3 > 0$ with the following. For large enough $A$ with probability at
  least $1-\alpha_3$ we have $\nu_m < \beta_3 A^2$.
\end{proposition}
\begin{proof}
  First we look at the hitting time analogous to $\tilde{\nu}_m$ for $Y'_u$ defined as $\hat{\nu}_m =
  \min \{s: Y'_u(s) \ge 1/2-\delta\}$.
  Without aiming for tight estimates $\hat{\nu}_m\le t$ can be ensured by
  $Y'_u(t)\ge 1/2-\delta$ and by Proposition \ref{prp:couple_yprime_z}
  this holds whenever $Z(t)\ge 1/2-\delta$. The latter is equivalent
  to $\tilde{Z}(t)\notin [-1/2+\delta,1/2-\delta]$.

  For some $\beta_3 >0$, the distribution of $\tilde{Z}(\beta_3A^2)$ is
  $\cN(X_u(0),\beta_3/2)$. Choosing $\beta_3$ large enough, the probability of
  this falling into
  $[-1/2+\delta,1/2-\delta]$ can be made below $\alpha_3/2$ and this event
  is a superset of $\hat{\nu}_m>\beta_3A^2$.

  Now apply Lemma \ref{lm:Yprime_coupling} with $\alpha_1 = \beta_3$. Note
  that $\tilde{\nu}_m \neq \hat{\nu}_m$ can only happen if $\nu_c^1 <\tilde{\nu}_m$.
  Also Lemma \ref{lm:nu_m_variants} ensures that $\nu_m$
  and $\tilde{\nu}_m$ almost always coincide.
  Altogether, we have $\nu_m = \tilde{\nu}_m = \hat{\nu}_m < \beta_3 A^2$ with an
  exceptional probability at most $O(\exp(-\beta_2 A^2)) + O(\exp(-\beta_1 A^2)) + \alpha_3/2$, this
  stays below $\alpha_3$ when $A$ is large enough, which completes the proof.
\end{proof}

\subsection{Diffusion from the middle}
\label{sec:mixmiddle}

In the previous subsection we have seen that the process $Y_u(t)$ eventually
has to reach the middle of the interval $[0,1]$ as formulated in
Proposition~\ref{prp:nu_m_bound}. Now we complement the analysis and consider the case
when the process is initialized from the middle, meaning $Y_u(0)
\in [1/2-\delta, 1/2+\delta]$.
Intuitively, we expect the process to evolve as a random walk with
independent Gaussian increments. However, we have to be careful as
boundary effects might alter the behavior of $Y_u(t)$ when it moves near the ends of the
interval $[0,1]$.
In this subsection we provide the techniques to estimate these
boundary effects which will allow to conclude that
the mixing of a random walk still translates to comparable mixing of $Y_u(t)$.

Let $W(t)$ be a random walk with i.i.d.\
$\cN(0,1/(2A^2))$ increments, starting from $W(0) := Y_u(0)$.
Our goal is to couple $W(t)$ with $Y_u(t)$ which only has a chance as
long as $W(t)$ stays within $[0,1]$.
\begin{definition}
  Let $\nu_c^2 := \min\{s:W(s)\notin [0,1]\}$.
\end{definition}

\begin{lemma}
  \label{lm:YWcoupling}
  There exist a coupling of the processes $Y_u$ and $W$ such that
  $Y_u(t)=W(t)$ whenever $t < \nu_c^2$.
\end{lemma}
\begin{proof}
  Assume the coupling holds until $t-1$, having $Y_u(t-1)=W(t-1)$.
  Let $\zeta \sim \cN(0,1/(2A^2))$ be independent from the past, then
  define $W(t) = W(t-1)+\zeta$. For
  $Y_u(t)$, accept $Y_u(t-1)+\zeta$ if it is in $[0,1]$ otherwise
  redraw it according to $\pi(\cdot,Y_u(t-1))$.

  The same values are obtained for the two processes at $t$ except
  if $W(t)$ is outside $[0,1]$. This is exactly the
  event we wanted to indicate with $\nu_c^2$ when we allow the
  two processes to decouple.
\end{proof}

\begin{lemma}
  \label{lm:intervalbound}
  For any $\alpha_4 > 0$ there exists $\beta_4>0$
  with the following property.
  For $A$ large enough, if $Y_u(0)\in
  [1/2-\delta, 1/2+\delta]$ there holds $P(\nu_c^2 < \alpha_4A ^2) <
  \beta_4$. We also have $\beta_4\ra 0$ as we choose $\alpha_4\ra 0$.
\end{lemma}
\begin{proof}
  We need to control the minimum and the maximum of a random walk
  where we use the following result of Erd\ho s and Kac \cite{erdos_kac:rw_max1946}: 
  \begin{theorem}[Erd\ho s-Kac]
    Let $\xi_1,\xi_,\ldots$ i.i.d.\ random variables, $\mE \xi_k = 0,~ D^2 \xi_k = 1$. Let
    $S_k = \xi_1 + \xi_2 + \ldots + \xi_k.$
    Then for any $\alpha \ge 0$
    $$\lim_{n\ra \infty} P(\max(S_1,S_2,\ldots,S_n) < \alpha \sqrt{n}) =
    \sqrt{\frac{2}{\pi}}\int_0^\alpha \exp(-x^2/2)dx.$$
  \end{theorem}
  Translating to the current situation, now that we use an initial value $Y_u(0)\in
  [1/2-\delta,1/2+\delta]$ as a reference, we want an upper bound on the
  probability that the partial sums generating $W(t)$ never exceed $1/2-\delta$ (nor they go
  below $-1/2+\delta$). The increments have variance $1/(2A^2)$ and the number
  of steps is $\alpha_4A^2$. Formally,
  \begin{align*}
    &P\left(\max(0,W(1)-W(0),\ldots, W(\alpha_4A^2)-W(0)) < 1/2-\delta\right) \\
    &= P\left(\max(0,W(1)-W(0),\ldots, W(\alpha_4A^2)-W(0))\sqrt{2}A <
    \frac{1-2\delta}{\sqrt{2\alpha_4}}\sqrt{\alpha_4A^2}\right) \\
    &\ra \sqrt{\frac{2}{\pi}}\int_0^{\frac{1-2\delta}{\sqrt{2\alpha_4}}}\exp(-x^2/2)dx.
  \end{align*}
  Now $\nu_c^2<\alpha_4A^2$ can only occur if this event fails and the
  maximum exceeds $1/2-\delta$, meaning $W(t)$ might exceed 1, or alternatively, the minimum of the process
  goes below $-1/2+\delta$ corresponding to $W(t)$ possibly leaving $[0,1]$
  at 0. Consequently, we may fix any small $\eps > 0$, then for any large enough $A$ we get
  \begin{equation}
    \label{eq:maxrw} 
    P(\nu_c^2 < \alpha_4 A^2) \le 2\left(1 -
      \sqrt{\frac{2}{\pi}}\int_0^{\frac{1-2\delta}{\sqrt{2\alpha_4}}}\exp(-x^2/2)dx
    \right) + \eps =: \beta_4.
  \end{equation}
  Observe that the right hand side of the expression indeed converges to
  0 as $\alpha_4\ra 0$.
\end{proof}

\begin{proposition}
  \label{prp:dist_Y_uniform}
  There exists a constant $\alpha_5 > 0$ such that for $A$ large enough, if $Y_u(0)\in
  [1/2-\delta, 1/2+\delta]$ we have 
  $$\|\cL(Y_u(\alpha_5A^2)) -  \pi_u\|_{\rm TV} < 1/3.$$
\end{proposition}
\begin{proof}
  We introduce $\alpha_5$ as a parameter. We will find sufficient
  conditions that ensure the claim of the proposition to hold, then pick
  a $\alpha_5$ that satisfies the conditions found.
  
  We first compare two simpler distributions, that of $W(\alpha_5A^2)$ and
  the uniform $\mu$. 
  By the definition of
  $W(t)$, the distribution of $W(\alpha_5A^2)$ is $\cN(Y_u(0),\alpha_5/2)$.
  $$\|\cL(W(\alpha_5A^2))-\mu\|_{\rm TV} =
  \frac{1}{2}\int_{-\infty}^\infty
  \left|\frac{\exp(-(x-Y_u(0))^2/\alpha_5)}{\sqrt{\alpha_5\pi}}-\mathbbm{1}_{[0,1]}(x)\right|dx$$
  The integrand has the form $|a-b|$ which we replace by
  $a+b-2\min(a,b)$ (knowing these variables are
  non-negative). Also, as the probability density functions integrate
  to 1, we get
  \begin{equation}
    \label{eq:W_mu_totalvar}
    \begin{aligned}
      &\|\cL(W(\alpha_5A^2))-\mu\|_{\rm TV} = 1 - \int_{-\infty}^\infty
      \min\left(\frac{\exp(-(x-Y_u(0))^2/\alpha_5)}{\sqrt{\alpha_5\pi}},\mathbbm{1}_{[0,1]}(x)\right)
      dx\\
      &= 1 - \int_0^1
      \min\left(\frac{\exp(-(x-Y_u(0))^2/\alpha_5)}{\sqrt{\alpha_5\pi}},1\right)
      dx\\
      &\le 1+2\delta - \int_{-\delta}^{1+\delta}
      \min\left(\frac{\exp(-(x-1/2)^2/\alpha_5)}{\sqrt{\alpha_5\pi}},1\right)
      dx =: \gamma.
    \end{aligned}
  \end{equation}
  The last inequality follows because the constant term is increased
  by $2\delta$, so is the length of the domain of the integration but
  the integrand is bounded above by 1. This step also involves an implicit change of
  variable depending on $Y_u(0)$, and it results in a final expression
  independent of this starting condition. The $\gamma$ we get is also
  independent of $A$, it does depend on $\delta$ but has a limit as
  $\delta\ra 0$.

  The claim of the lemma is about two other distributions, now we relate them to
  the ones just compared. Using Lemma~\ref{lm:intervalbound} for
  $\alpha_4=\alpha_5$ we know that
  $Y_u(t)$ and $W(t)$ can be coupled well up to $t=\alpha_5A^2$, which
  directly implies
  \begin{equation}
    \label{eq:Y_W_totalvar}
    \|\cL(Y_u(\alpha_5A^2))-\cL(W(\alpha_5A^2))\|_{\rm TV} \le \beta_4,
  \end{equation}
  where $\beta_4$ is the constant given by Lemma~\ref{lm:intervalbound}.

  To compare $\pi_u$ with $\mu$ we show $\pi_u$ converges to $\mu$ in
  total variation as $A\ra\infty$.
  For every $x\in [0,1]$ define
  $$p_u(x) = \frac{A}{\sqrt{\pi}}\int_{-x}^{1-x}\varphi(y)dy,$$
  this is a
  function proportional to the density of $\pi_u$. By standard
  Gaussian tail estimates for all $x\in (0,1)$ we get
  $$1 - \frac{\exp(-A^2x^2)}{2\sqrt{\pi}Ax} -
  \frac{\exp(-A^2(1-x)^2)}{2\sqrt{\pi}A(1-x)} \le p_u(x) \le 1.$$
  Hence for all $x\in (0,1),~p_u(x)\ra 1$ as $A\ra\infty$. These
  are uniformly bounded functions, so $\int_0^1p_u\ra 1$. The
  expression to consider for the convergence of the distributions is
  $$\|\mu - \pi_u\|_{\rm TV} = \frac{1}{2}\int_0^1
  \left|1-\frac{p_u(x)}{\int_0^1 p_u}\right|dx.$$
  Here $1/\int_0^1 p_u$ is converging to 1 and is therefore bounded
  after some threshold, so the functions are eventually uniformly bounded and pointwise converging to
  0. Thus the integrals also converge, and we get
  \begin{equation}
    \label{eq:pi_mu_limit}
    \lim_{A\ra\infty}\|\mu - \pi_u\|_{\rm TV} = 0.
  \end{equation}

  We can now combine our bounds of \eqref{eq:W_mu_totalvar},
  \eqref{eq:Y_W_totalvar} and \eqref{eq:pi_mu_limit}:
  \begin{align*}
    \|\cL(Y_u(\alpha_5A^2)) -  \pi_u\|_{\rm TV} &\le \|\cL(Y_u(\alpha_5A^2)) -
                                              \cL(W(\alpha_5A^2))\|_{\rm TV} + \|\cL(W(\alpha_5A^2))-\mu\|_{\rm TV}\\
                                            &+\|\mu - \pi_u\|_{\rm TV}
                                              < \beta_4 + \gamma + \eps,
    \end{align*}
  where $\eps>0$ can be as small as wanted by setting $A$ large
  enough. The proposition holds if we can ensure this sum to be small enough.

  Note that a strong compromise is present for the choice of
  the constant $\alpha_5$. In \eqref{eq:Y_W_totalvar} we want to limit how likely the boundaries of the
  unit interval are to be reached, at the same time in \eqref{eq:W_mu_totalvar} we want to show that
  $Y_u(s)$ is already spread out to some extent.

  Still, a specific choice is possible. For $\alpha_5 = 0.10$ Lemma
  \ref{lm:intervalbound} provides $\beta_4\approx 0.051$ when using
  $\delta=\eps=0$ and computer calculations for \eqref{eq:maxrw}. By choosing $\delta,
  \eps >0$ but small enough, trusting computers but not too much, we
  can safely say $\beta_4 < 0.06$. In \eqref{eq:W_mu_totalvar} using
  the same choice of $\alpha_5$ we numerically get $\gamma \approx 0.263$ for
  $\delta=\eps=0$. Once again we allow a
  safety margin to only claim $\beta_4+\gamma+\eps<1/3$.
\end{proof}

\section{Overall mixing}
\label{sec:mixproof}

We are now ready to establish mixing time bounds for the process we
understand the best, $Y_u(t)$, then we will translate those results to
the original process of interest $X(t)$.

Let us define
$$
d(t) := \sup_{Y_u(0)\in[0,1]} \|\cL(Y_u(t))-\pi_u\|_{\rm TV},
$$
which measures the distance from the stationary distribution from the
worst starting point. We can give good bounds based on the previous
sections:

\begin{lemma}
  \label{lm:d_decrease}
  There exists a constant $\beta_6>0$ such that $d(\beta_6 A^2) < 4/9$.
\end{lemma}
\begin{proof}
  Intuitively, from any starting point we can first wait for the
  process to reach the middle and then let the diffusion happen from there, as these
  are components we can already control.

  Let us apply
  Proposition~\ref{prp:nu_m_bound} with $\alpha_3 = 1/9$ providing a
  certain $\beta_3$. Once the process is in the middle part
  $[1/2-\delta,1/2+\delta]$ we know by
  Proposition~\ref{prp:dist_Y_uniform} that in the subsequent $\alpha_5A^2$
  steps sufficient diffusion occurs. Let $\beta_6=\beta_3+\alpha_5$.

  Formally, fix $Y_u(0)\in [0,1]$. We perform our calculations by
  conditioning on the value of $\nu_m$.
  $$
  \|\cL(Y_u((\beta_3+\alpha_5)A^2))-\pi_u \|_{\rm TV} =
  \left\|\sum_{s=0}^{\infty}P(\nu_m=s)
    \cL(Y_u((\beta_3+\alpha_5)A^2)~|~\nu_m=s) - \pi_u \right\|_{\rm TV}.
  $$
  Conditioned on $\nu_m=s$, $Y_u(s)\in [1/2-\delta,1/2+\delta]$,
  therefore Proposition~\ref{prp:dist_Y_uniform} provides
  $\|\cL(Y_u(s+\alpha_5A^2)~|~\nu_m=s)-\pi_u\|_{\rm TV} < 1/3$. We use this for $s\le
  \beta_3A^2$, then performing $\beta_3A^2-s$
  more steps can only decrease this distance, see
  \cite[Chapter~4]{levin:2009markov} for a detailed discussion about this. For $s>\beta_3A^2$ we use
  the trivial bound on the total variation distance. We get
  $$
  \|\cL(Y_u((\beta_3+\alpha_5)A^2))-\pi_u \|_{\rm TV} \le \sum_{s=0}^{\beta_3A^2} P(\nu_m=s) \cdot \frac{1}{3}+
  P(\nu_m>\beta_3A^2)\cdot 1 \le \frac{1}{3} + \alpha_3 = \frac{4}{9}.
  $$
\end{proof}

A slight variation of $d(t)$ compares the distribution of the process
when launched from two different starting points.
$$
\bar{d}(t) := \sup_{Y^1_u(0),Y^2_u(0)\in[0,1]} \|\cL(Y^1_u(t))-\cL(Y^2_u(t))\|_{\rm TV},
$$
Standard results provide the inequalities $d(t) \le \bar{d}(t) \le 2d(t)$ and the
submultiplicativity $\bar{d}(s+t)\le\bar{d}(s)\bar{d}(t)$, see
\cite[Chapter~4]{levin:2009markov}. The results therein are given for finite
state Markov chains
but are straightforward to translate to the current case of absolutely
continuous distributions and transition kernels.

\begin{proposition}
  \label{prp:Yu_mixing}
  For any $0<\alpha_7<1$ there exists $\beta_7>0$ such that
  $$t_{\rm mix}(Y_u, \alpha_7) \le \beta_7 A^2.$$
\end{proposition}
\begin{proof}
  Using Lemma~\ref{lm:d_decrease} for any $k\ge 1$ we get
  $$d(k \beta_6 A^2) \le \bar{d}(k \beta_6 A^2) \le (\bar{d}(\beta_6 A^2))^k
  \le (2d(\beta_6 A^2))^k \le \left(\frac{8}{9}\right)^k.$$
  For $k = \lceil\log \alpha_7 / \log (8/9)\rceil$ this is less than
  $\alpha_7$ thus by setting
  $\beta_7 = \beta_6 \lceil \log \alpha_7 / \log (8/9)\rceil$  the
  process will be close enough to the stationary distribution as
  required at $t=\beta_7A^2$.
\end{proof}

\begin{lemma}
  \label{lm:Yu_Y_samemixing}
  The mixing time of $Y_u$ and $Y$ are nearly the same, for any
  $0<\alpha_7<1$
  $$t_{\rm mix}(Y, \alpha_7) = t_{\rm mix}(Y_u, \alpha_7) + 1.$$
\end{lemma}
\begin{proof}
  First, we use that the total variation distance between the marginals is at
  most the distance between the overall distributions. Consequently,
  for any $t$ we have $\|\cL(Y_u(t-1))-\pi_u\|_{\rm TV} \le
  \|\cL(Y(t))-\pi\|_{\rm TV}$. This gives $t_{\rm mix}(Y, \alpha_7)
  \ge t_{\rm mix}(Y_u, \alpha_7) + 1$.

  For the other direction, assume $\|\cL(Y_u(t))-\pi_u\|_{\rm TV} \le
  \alpha_7$ for some $t$. This means there is an optimal coupling with a
  random variable $\tilde{Y}_u^1$ having distribution $\pi_u$ such that
  $P(Y_u(t) \neq \tilde{Y}_u^1) \le \alpha_7$. As $\tilde{Y}_u^1$ has
  distribution $\pi_u$, it is possible to draw an additional random
  variable $\tilde{Y}_u^2$ to get
  $(\tilde{Y}_u^2,\tilde{Y}_u^1)$ with distribution $\pi$.

  This is the same step when generating $Y_u(t+1)$ from $Y_u(t)$ thus
  we may keep the above coupling whenever already present. Therefore we have
  $P\big((Y_u(t+1),Y_u(t)) \neq (\tilde{Y}_u^2,\tilde{Y}_u^1)\big)\le
  \alpha_7$ which can also be written as $\|\cL(Y(t+1))-\pi\|_{\rm TV} \le
  \alpha_7$. This implies $t_{\rm mix}(Y, \alpha_7) \le t +1$, completing the proof.  
\end{proof}

We are now ready to prove the main theorem of the paper, as stated
in the introduction.

\newtheorem*{thm:Xmix}{Theorem \ref{thm:Xmix}}
\begin{thm:Xmix}
  Let $X(t)$ be the coordinate Gibbs sampler for the diagonal
  distribution. For any $0<\alpha<1$ there exists $\beta > 0$ such that for large
  enough $A$
  $$t_{\rm mix}(X, \alpha) \le \beta A^2.$$
\end{thm:Xmix}
\begin{proof}
  We use Proposition~\ref{prp:Yu_mixing} with $\alpha_7=\alpha/2$ and get a
  constant $\beta_7$ such that $t_{\rm mix}(Y_u,\alpha/2)\le \beta_7A^2$ and by
  Lemma~\ref{lm:Yu_Y_samemixing} also $t_{\rm mix}(Y, \alpha/2)\le
  \beta_7A^2 + 1$. At each step the distribution of $X^*$ and $Y$ might differ
  only by flipping along the diagonal, which does not change the
  distance from the (symmetric) $\pi$ thus also leaves the mixing time
  the same so we get $t_{\rm mix}(X^*, \alpha_7/2)\le \beta_7A^2 +1$.

  The definition of $X^*$ was based on the observation that when the
  same coordinate is rerandomized repeatedly, no additional mixing
  happens and the values at that coordinate simply get
  overwritten. Let us now quantify this effect, counting how many
  times did the direction of randomization change:
  $$
  N(t):=|\{s~:~1\le s \le t-1,~r(s)\neq r(s+1)\}|.
  $$
  With this notation we see that $\cL(X(t) ~|~ N(t)=k,r(1)=V) =
  \cL(X^*(k+1))$ for all $t\ge 1$.

  Without the loss of generality we now focus on the case of
  $r(1)=V$. Let us express the distribution of $X(t)$ conditioning on
  the value of $N(t)$.
  $$
  \cL(X(t)~|~r(1)=V) = \sum_{k=0}^{t-1}P(N(t)=k)\cL(X^*(k+1)) =
  \sum_{k=0}^{t-1}\frac{1}{2^{t-1}}{t-1 \choose k}\cL(X^*(k+1)).
  $$
  We substitute $t=3\beta_7 A^2$ and evaluate the total variation
  distance from $\pi$.  
  \begin{align*}
    &\|\cL(X(3\beta_7 A^2)~|~r(1)=V)-\pi\|_{\rm TV} \le
    \sum_{k=0}^{t-1}\frac{1}{2^{t-1}}{t-1 \choose k} \|\cL(X^*(k+1)) -
    \pi\|_{\rm TV}\\
    &\le P(Binom(3\beta_7 A^2-1,1/2) < \beta_7A^2) \cdot 1 + 1
      \cdot \|\cL(X^*(\beta_7A^2+1)) - \pi\|_{\rm TV}\\
    &\le \exp(-\eps\beta_7 A^2) + \frac{\alpha}{2}.
  \end{align*}
  The last line holds with some positive $\eps$ by Hoeffding's inequality for the Binomial
  distribution and by substituting the upper bound on the total variation
  distance when we know $k$ is above the mixing time.
  For large enough $A$ this is below $\alpha$.

  By symmetry, the same bound holds for $\cL(X(3\beta_7A^2)~|~r(1)=U)$ and by
  convexity it is also true for the mixture of the two, the
  unconditional distribution of $X(3\beta_7A^2)$. This concludes the proof
  with $\beta=3\beta_7$.
\end{proof}

Finally, let us comment on the multitude of constants
$\alpha_i,\beta_i$ appearing throughout the proofs, verifying that they
can be consistently chosen when needed. First, a small enough $\delta>0$
has to be picked for the proof of Proposition
\ref{prp:dist_Y_uniform} which also relies on Lemma \ref{lm:intervalbound}. Once
it is fixed, observe that in the remaining Sections \ref{sec:middle}
and \ref{sec:mixproof} all the constants only depend on other ones with lower indices,
with the last $\alpha, \beta$ of Theorem \ref{thm:Xmix} also depending on some previous ones.
This excludes the issue of circular dependence.

\section{Further estimates}
\label{sec:others}

In this section we complement the main result Theorem~\ref{thm:Xmix}
by a lower bound showing that the order of $A^2$ is exact and by
demonstrating the evolution of the distribution via numerical simulations.

Such a lower bound is plausible once having Lemma~\ref{lm:YWcoupling} and
Lemma~\ref{lm:intervalbound}, these roughly say that when starting from the
middle $Y_u$ behaves like a random walk for order of $A^2$ steps and
reaches only constant distance in order of $A^2$ steps. Let us proceed
by forming a formal argument.

\begin{theorem}
  \label{thm:lowerbound}
  Let $X(t)$ be the coordinate Gibbs sampler for the diagonal distribution.
  There exists constants $\alpha', \beta' > 0$ such that for large
  enough $A$
  $$ t_{\rm mix}(X,\alpha') > \beta'A^2. $$
\end{theorem}

First of all, to bound the mixing time from below it is sufficient
to give a lower bound on the number of steps needed for a single starting point. In this spirit,
we set $X(0) = (1/2,1/2)$. With this choice, the arguments in
Section~\ref{sec:mixmiddle} can be applied.

Set $S = [0,1/4]^2 \cup [3/4,1]^2$. Once we prove
$\pi(S)-P(X(\beta'A^2)\in S) > \alpha'$ for a proper choice of
$\alpha',\beta'$
that warrants a large total variation distance at the time $\beta'A^2$
and confirms our bound
for the mixing time.

\begin{lemma}
  \label{lm:bigS}
  $\pi(S) \ge 1/8.$
\end{lemma}
\begin{proof}
  If we divide the unit square to 4-by-4 equal size smaller squares,
  then $S$ is composed of two of these smaller squares, see Figure~\ref{fig:sqsplit}. 
  \begin{figure}[h]
    \centering
    \includegraphics[width=0.4\textwidth]{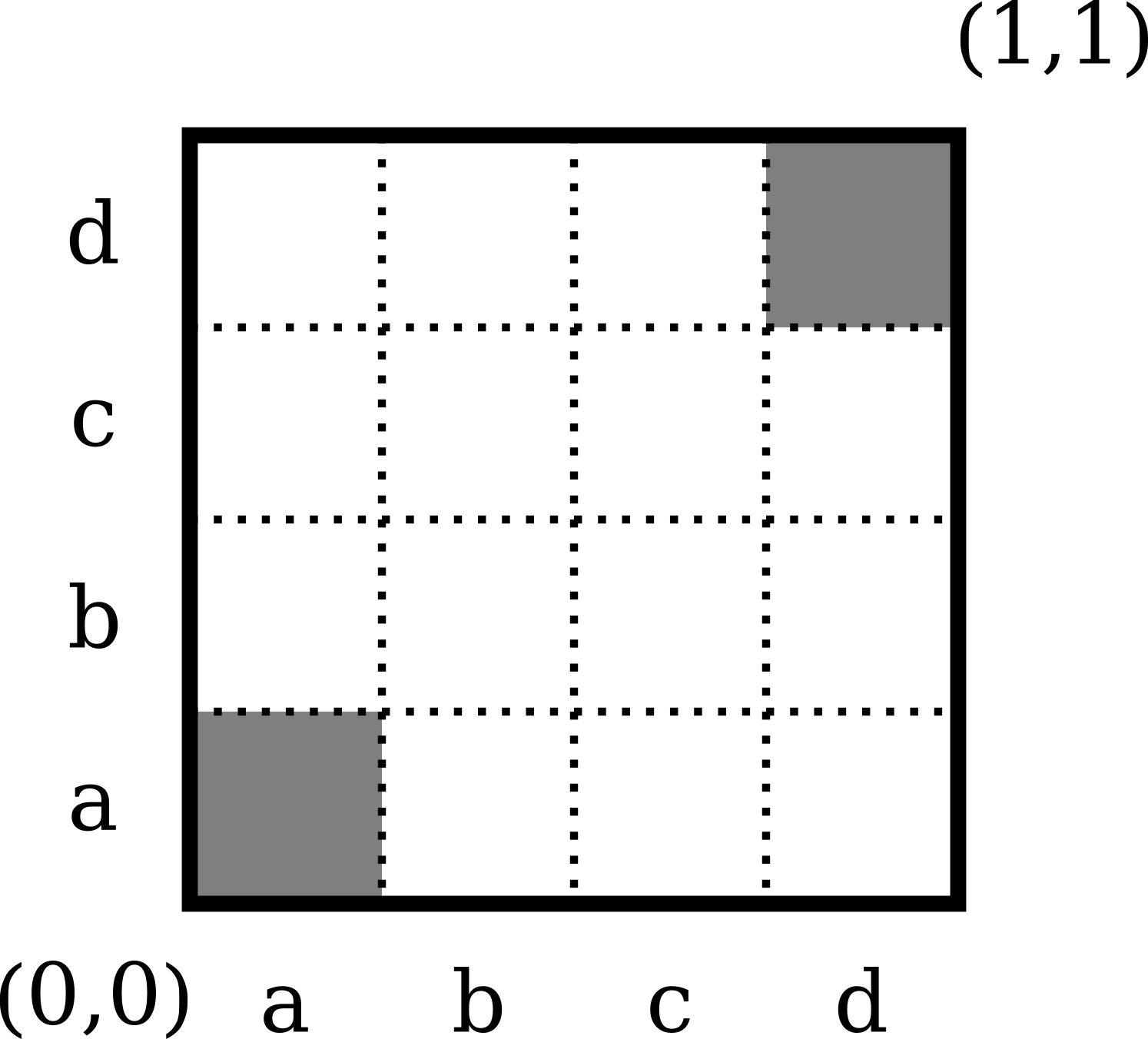}
    \caption{4-by-4 division of the unit square to smaller
      squares. Horizontal and vertical intervals are labeled with
      letters for easier reference. The shaded area represents $S$.}
    \label{fig:sqsplit}
  \end{figure}
  It is enough to show that the selected squares forming $S$ have greater or
  equal probability than the other squares w.r.t.\ $\pi$, this directly
  confirms $\pi(S)\ge 2/16 = 1/8$.

  To verify this, we compare the
  unnormalized density $\varphi$ on them. We use the simple
  inequality that for $u,v\in [0,1/4]$ and any $x\ge 0$ we have
  $$
  \varphi(u-v) \ge \varphi(u-(1/2-v+x)).
  $$
  Indeed, note that $\varphi(y)$ is monotone decreasing with
  $|y|$. Then for $u\ge v, x=0$ an easy comparison of the arguments
  provides the bound, while the other cases follow similarly.
  Observe that for $x=0$ this inequality compares $\varphi$ at some point and
  its reflection to the line $v=1/4$. Setting $x>0$ corresponds to a
  further shift increasing the $v$ coordinate after the reflection.

  Using this we see that the points of the small square labeled by
  $(a,a)$ in Figure~\ref{fig:sqsplit} correspond to
  the points of $(a,b)$ after a reflection so $\varphi$ is pointwise
  larger on $(a,a)$ by the above inequality. The same comparison holds against $(a,c),~(a,d)$
  where an additional shift is necessary besides the reflection.
  Consequently, $\pi$ is maximal for the $(a,a)$ square compared to the
  others in its column.

  Additionally, note that $\varphi(u-v)$ is invariant under the shift
  of $(u,v)$ in the direction $(1,1)$. Therefore $\pi$ is exactly the
  same for the four squares on the diagonal. Furthermore, all the
  other squares are diagonally shifted and/or reflected (w.r.t.\ the
  diagonal) copies of the
  ones considered above, where we have seen that their probability is upper
  bounded by the probability of the square $(a,a)$. The distribution $\pi$ is
  symmetric w.r.t.\ the diagonal, so we conclude that $(a,a)$
  (and therefore also $(d,d)$) have indeed maximal probability among
  all squares.
\end{proof}

\begin{lemma}
  \label{lm:smallS}
  For any $\alpha'_1>0$ there exists $\beta'_1>0$ such that for large
  enough $A$ any $t\le \beta'_1 A^2$ satisfies
  $$P(Y(t) \in S) < \alpha'_1.$$
\end{lemma}
\begin{proof}
  We want to rely on the previous observations that $Y_u(t)$ behaves
  like a random walk for a while with certain Gaussian
  increments. Using Lemma~\ref{lm:intervalbound} we can choose
  $\alpha_4>0$ so that the corresponding $\beta_4$ goes below
  $\alpha'_1/2$. Let us denote this $\alpha_4$ by $\beta'_2$ for convenience.

  Also, there exists $\beta'_3>0$ so that
  $$P(\cN(1/2, \beta'_3/2) \in [0,1/4]\cup [3/4,1]) < \alpha'_1/2,$$
  and clearly the same probability bound holds if the variance is
  decreased. Fixing $Y_u(0)=W(0)=1/2$, the distribution of $W(\beta'_3
  A^2)$ is exactly $\cN(1/2, \beta'_3/2)$.

  To join our estimates we form
  $$P(Y(t) \in S) \le P(Y_u(t) \in [0,1/4]\cup [3/4,1]) \le
  P(W(t)\neq Y_u(t)) + P(W(t) \in [0,1/4]\cup [3/4,1]).$$
  For $t\le \beta'_2 A^2$ the first term is below $\alpha'_1/2$ as it
  is an upper bound for the decoupling of $W,Y_u$ to happen. For $t\le
  \beta'_3 A^2$ the second term is below $\alpha'_1/2$. Altogether, if $t
  \le \min(\beta'_2,\beta'_3) A^2$,
  $$P(Y(t) \in S) < \alpha'_1.$$
  Therefore by choosing $\beta'_1 = \min(\beta'_2,\beta'_3)$ we
  complete the proof.  
\end{proof}

\begin{proof}[Proof of Theorem~\ref{thm:lowerbound}]
  Apply Lemma~\ref{lm:smallS} with $\alpha'_1 = 1/16$ to get some
  $\beta'_1$. The distribution of $X(\beta'_1A^2)$ is a mixture of the
  distributions of $Y(t)$ and their diagonally flipped version for
  $t\le \beta'_1A^2$, where $t$ corresponds to how many times the
  rerandomization happened in a new direction. The set $S$ is
  symmetric w.r.t.\ the diagonal so for $P(X(\beta'_1 A^2)\in S)$ we can
  simply say it is a convex combination of $P(Y(t)\in S),~t\le
  \beta'_1 A^2$ without needing any correction for the diagonal flip. Now by
  Lemma~\ref{lm:smallS} each of these probabilities are below
  $\alpha'_1$, therefore it follows that
  $$P(X(\beta'_1 A^2)\in S) < \alpha'_1 = \frac{1}{16}.$$

  Comparing this with the statement of Lemma~\ref{lm:bigS} we get
  $$\pi(S) - P(X(\beta'_1 A^2)\in S) > \frac{1}{16}.$$
  Consequently $\|\cL(X(\beta'_1 A^2)) - \pi\|_{\rm TV} > 1/16$, so
  $t_{\rm mix}(X,1/16) > \beta'_1 A^2$. Thus the theorem holds with
  the choice $\alpha'=1/16,~\beta'=\beta'_1$.
\end{proof}

Finally, we present numerical approximations of the evolution of the
distribution over time for different values of $A$. The unit square is
discretized with a resolution of $500\times 500$ and the distribution
is calculated along these points. The starting point is always $(0,0)$
at the lower left corner. The results are presented in
Figure~\ref{fig:dist_sim} for different $A$ and different $t$. Both
the convergence to the stationary distribution is visible and also how
this distribution becomes more concentrated along the diagonal for
higher values of $A$. We also computed the time necessary to get within a
total variation distance of $1/4$ of the stationary distribution, for
$A=10,t=71$, for $A=50, t=1858$, for $A=250, t=47233$ is needed. This
is a good proxy for the mixing time, note that only a single (but
intuitively bad) starting point was tested and the discretization
might have introduced some error. Still, the quadratic growth of $t$
with respect to the increase of $A$ is already apparent.

\begin{figure}[h]
    \centering
    \subfloat[$A=10,~t=100$]{
      \includegraphics[width=0.25\textwidth]{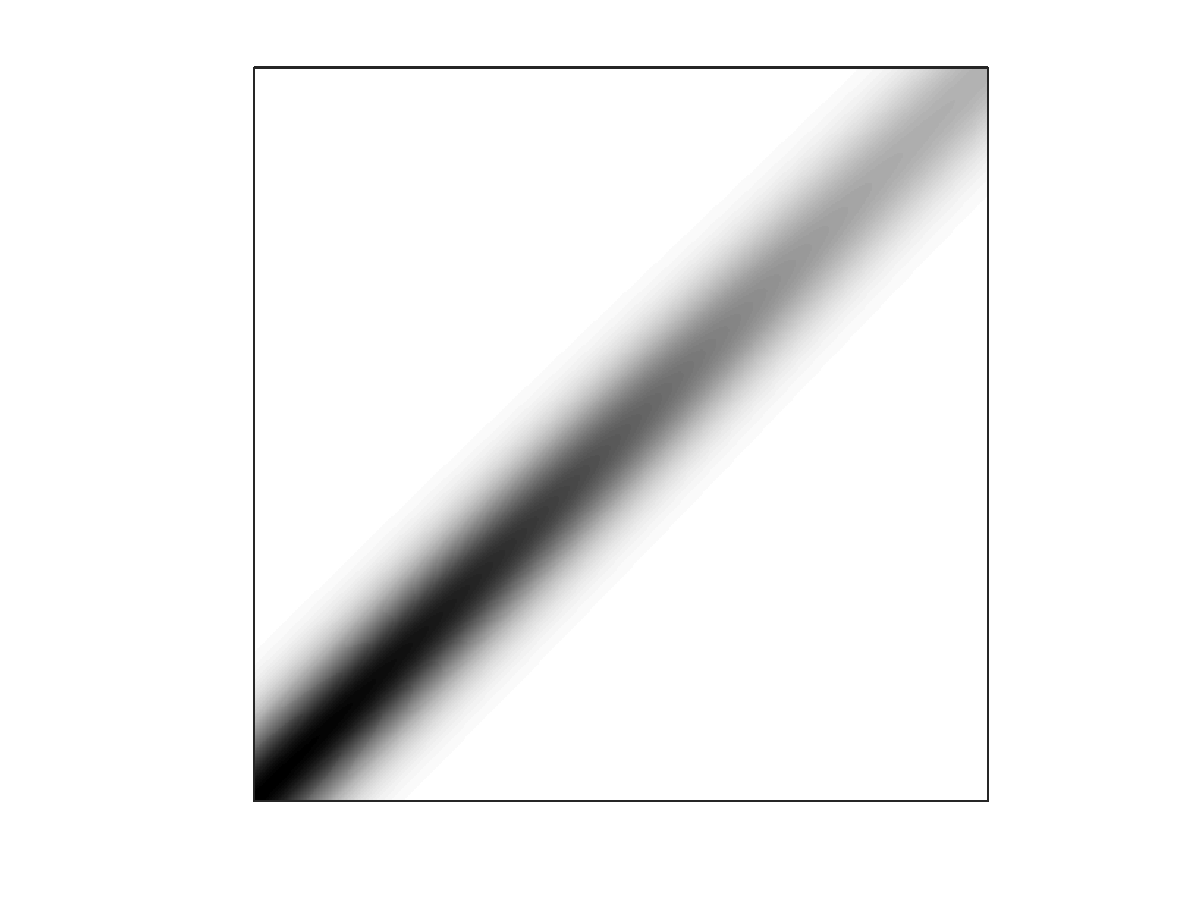}
    }
    \subfloat[$A=10,~t=1000$]{
        \includegraphics[width=0.25\textwidth]{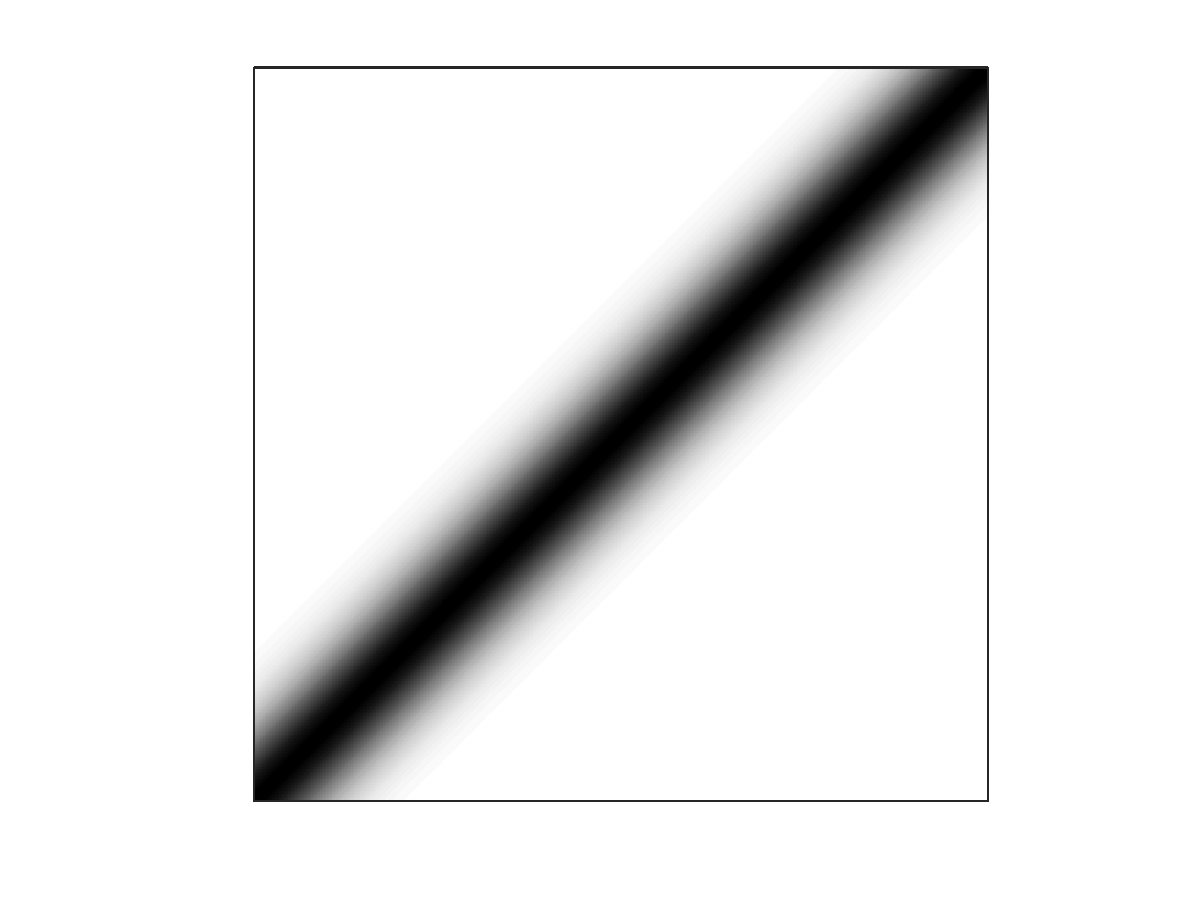}
      }
    \subfloat[$A=10,~t=10000$]{
      \includegraphics[width=0.25\textwidth]{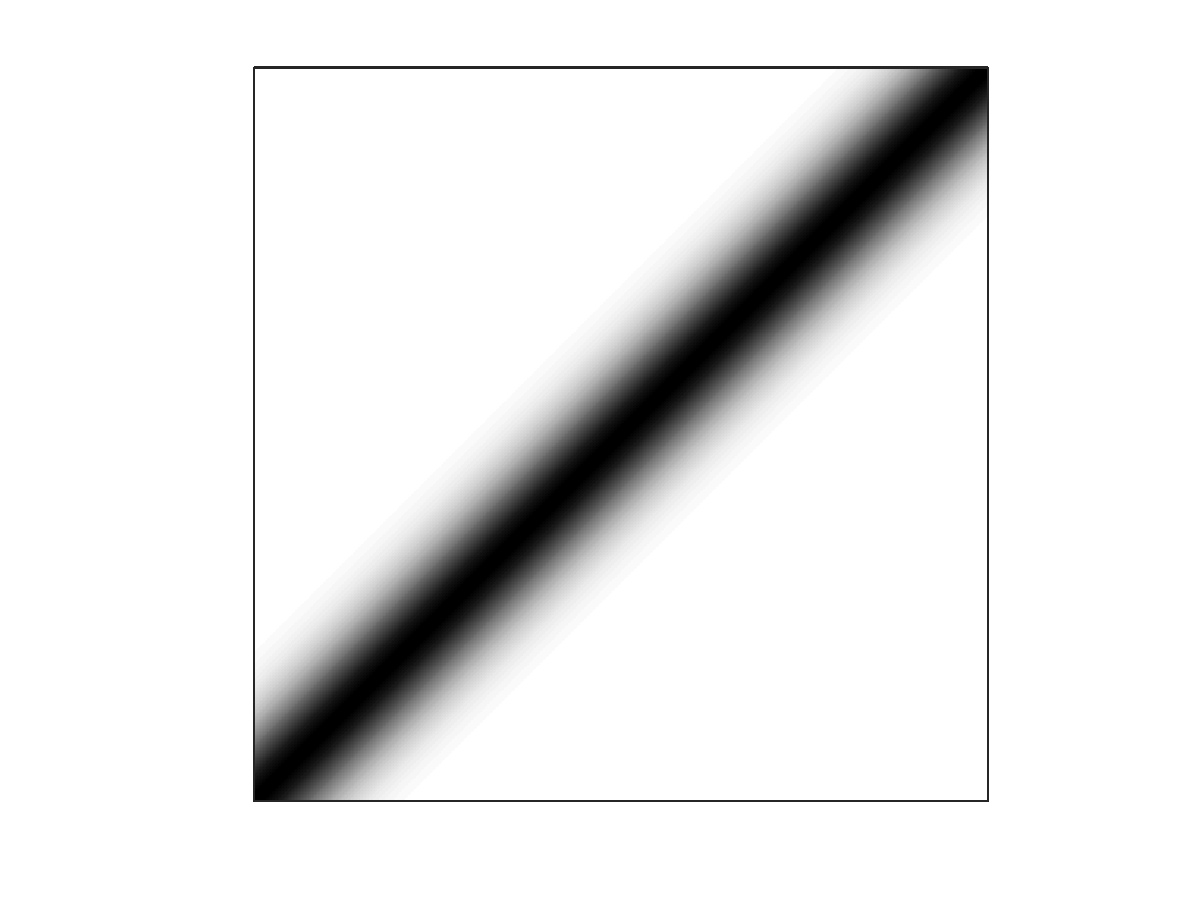}
    }

    \subfloat[$A=50,~t=100$]{
      \includegraphics[width=0.25\textwidth]{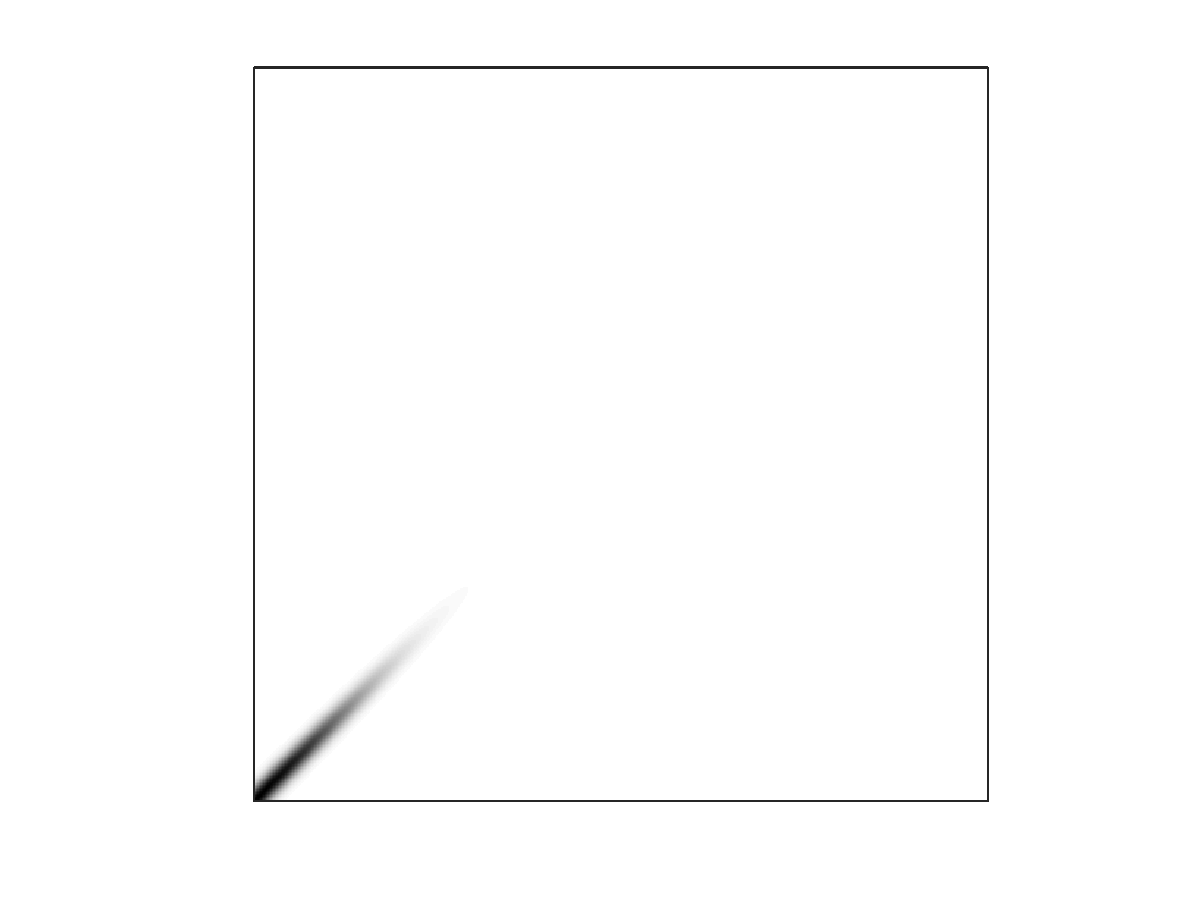}
    }
    \subfloat[$A=50,~t=1000$]{
      \includegraphics[width=0.25\textwidth]{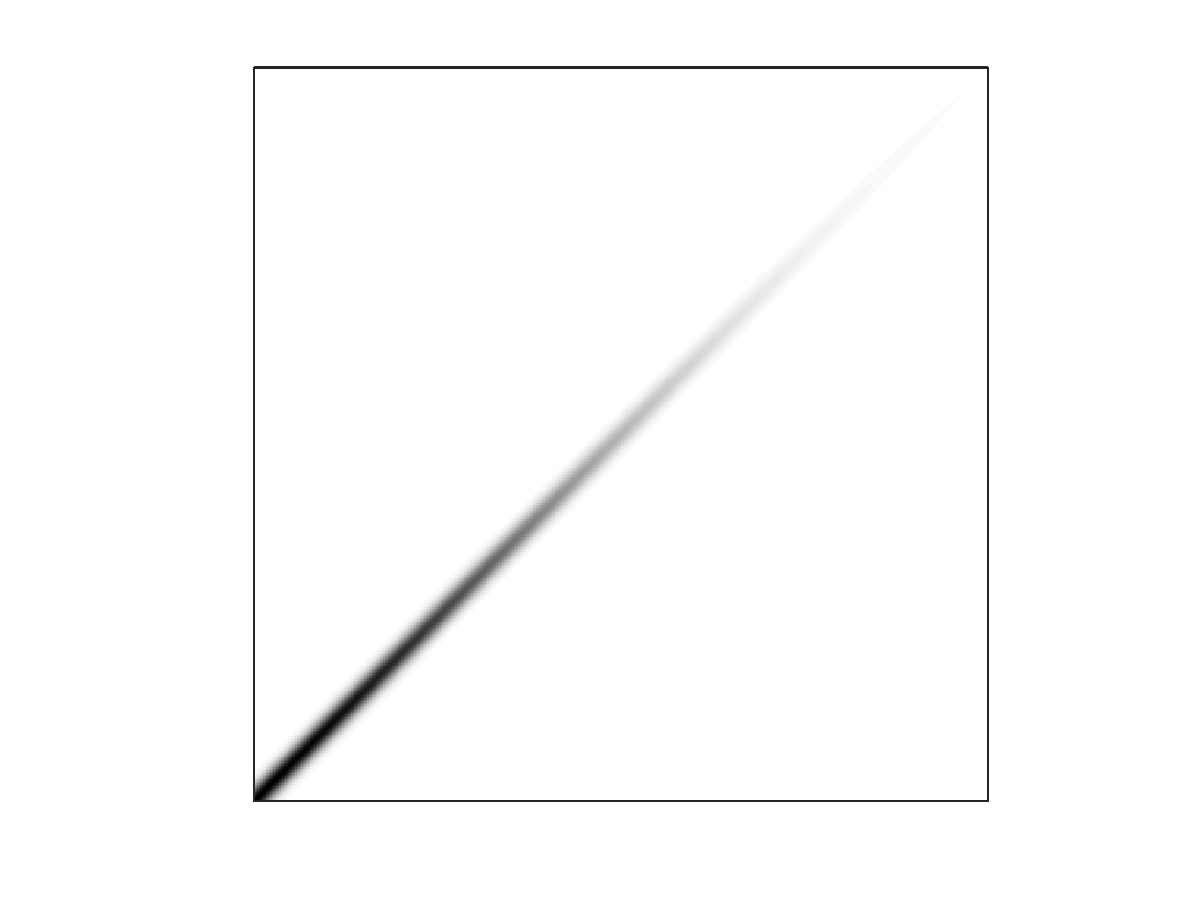}
    }
    \subfloat[$A=50,~t=10000$]{
      \includegraphics[width=0.25\textwidth]{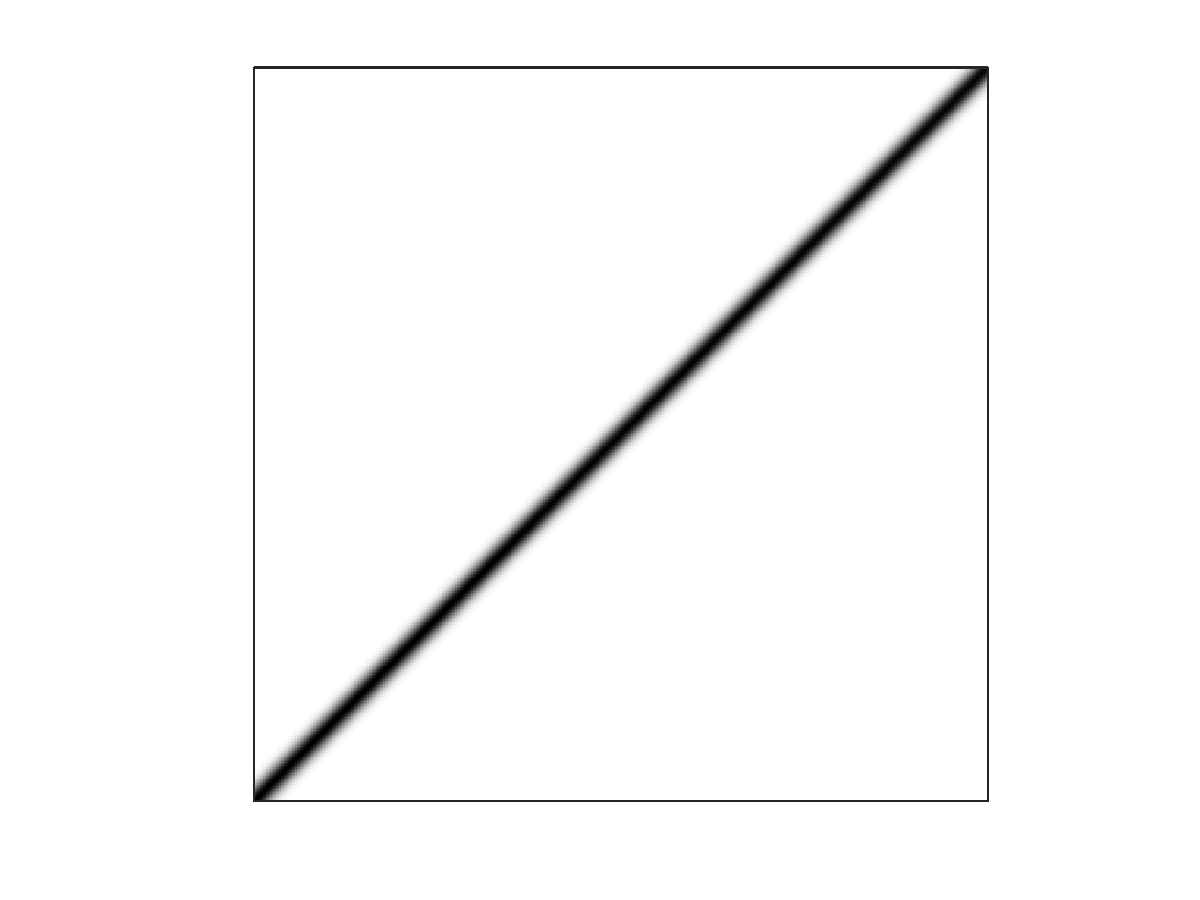}
    }
    
    \subfloat[$A=250,~t=100$]{
      \includegraphics[width=0.25\textwidth]{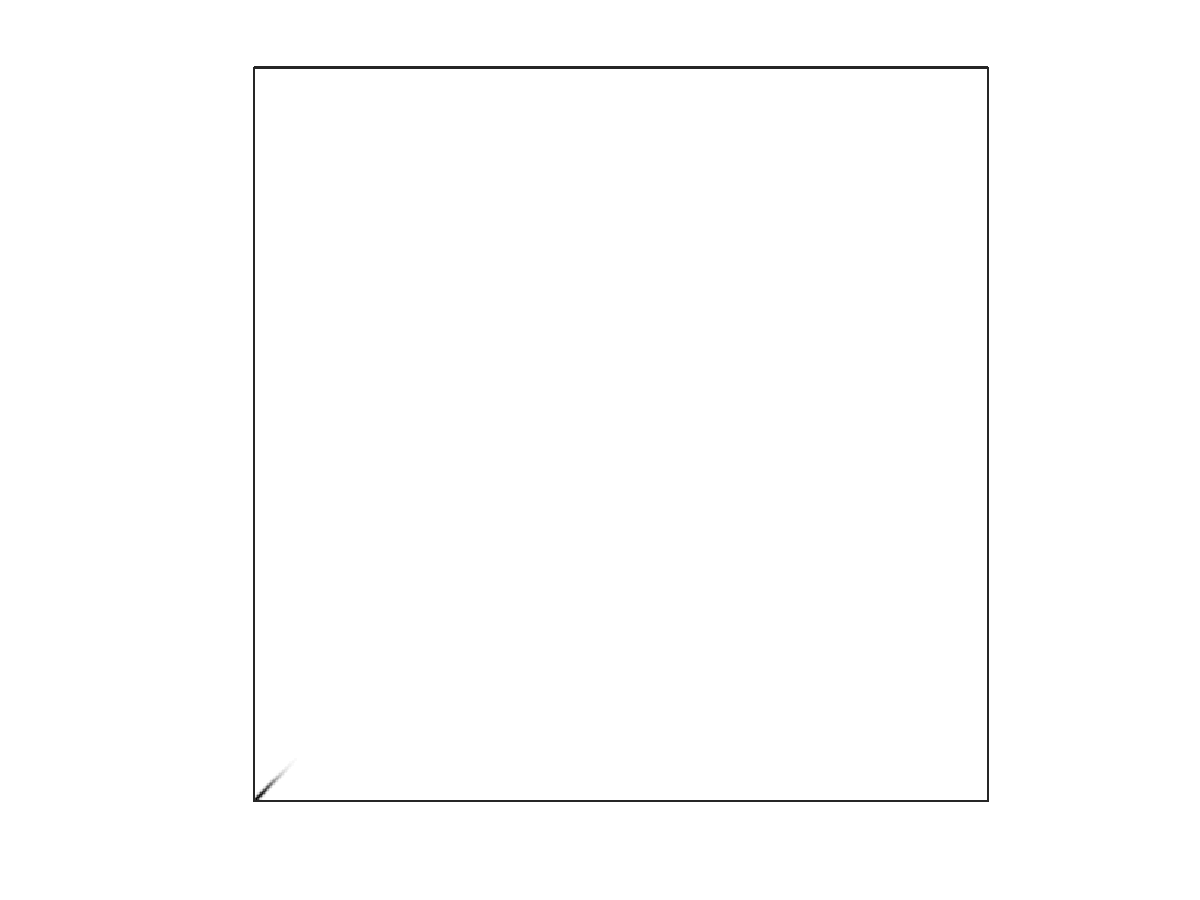}
    }
    \subfloat[$A=250,~t=1000$]{
      \includegraphics[width=0.25\textwidth]{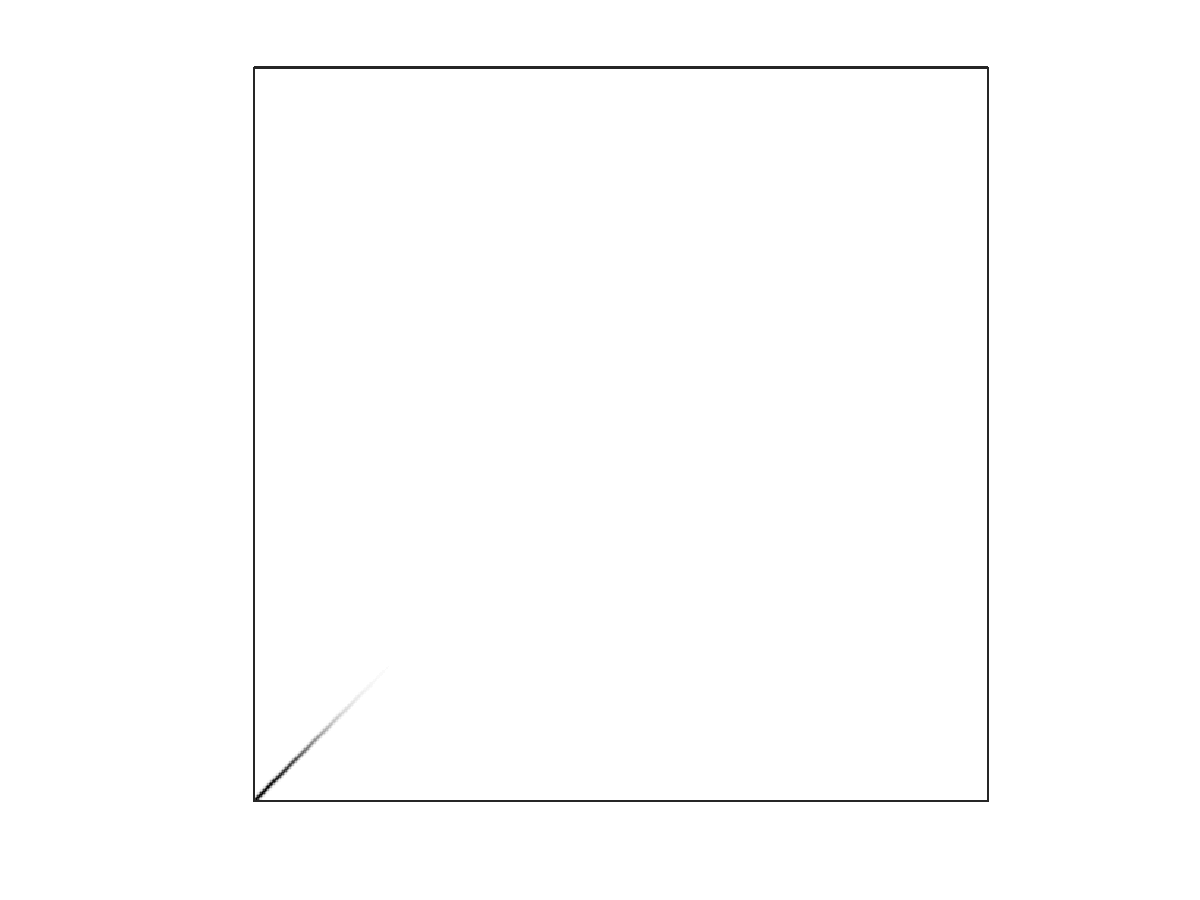}
    }
    \subfloat[$A=250,~t=10000$]{
      \includegraphics[width=0.25\textwidth]{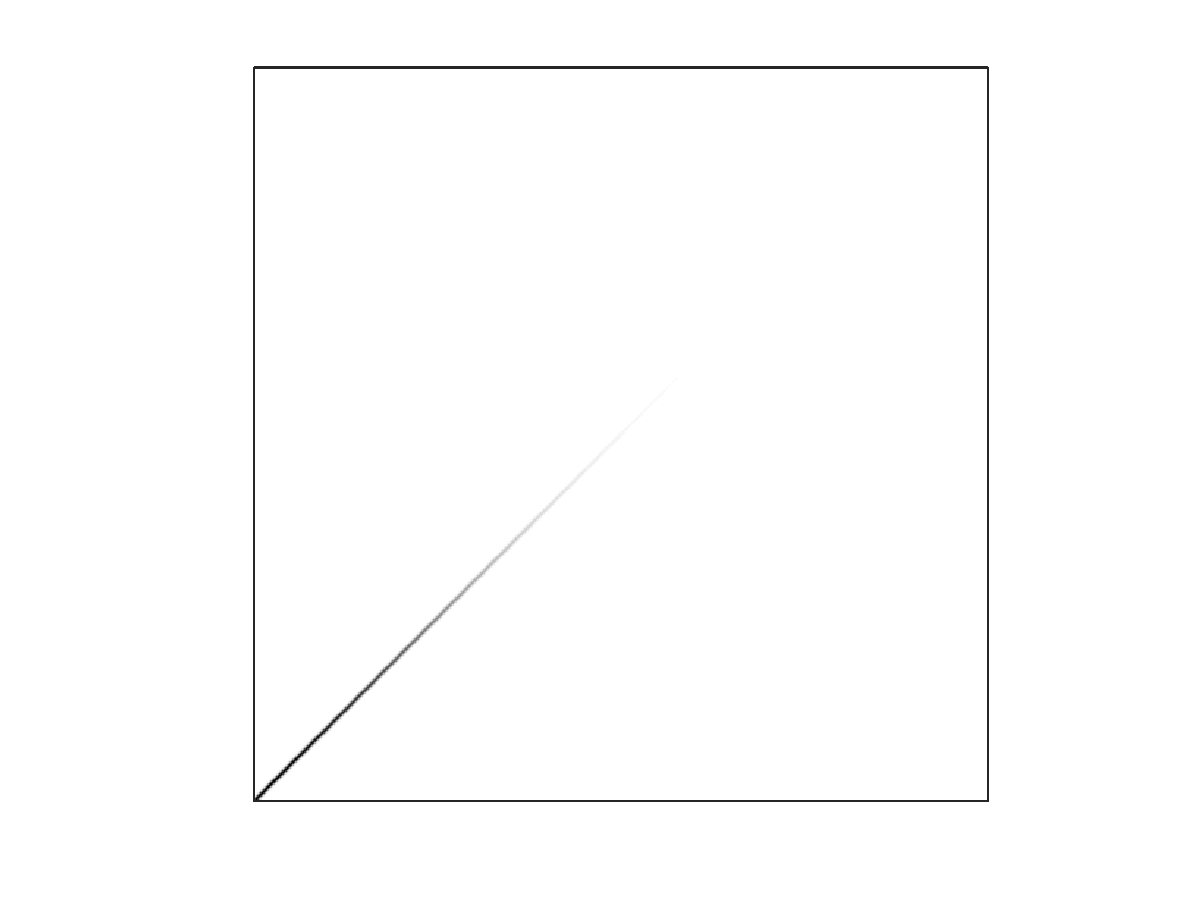}
    }

    \caption{Density of $\cL(X(t))$ for different parameters $A$ and
      $t$. Darker colors represent higher values (shades scale
      individually for each image).}
    \label{fig:dist_sim}
\end{figure}

\section*{Acknowledgments}
The author would like to express his thanks to Persi Diaconis and
Gy\"orgy Michaletzky for their inspiring
comments and to the American Institute of Mathematics for the
stimulating workshop they hosted and organized.

\bibliographystyle{siam}
\bibliography{mcmt}

\end{document}